\documentclass[]{article}
\usepackage{geometry}
\usepackage{amsthm}
\usepackage{mathrsfs} 
\usepackage{amsmath,amssymb}
\usepackage{indentfirst}
\newtheorem{theorem}{Theorem}[section]  

\newtheorem{proposition}[theorem]{Proposition}
\newtheorem{lemma}[theorem]{Lemma}

\newtheorem{remark}[theorem]{Remark}

\usepackage{color}
\usepackage{hyperref}
\hypersetup{
	colorlinks=true,
	linkcolor=blue,
	filecolor=magenta,
	urlcolor=cyan,
	citecolor=red
}
\begin{document}
\title{Exploring the Representation of Large Positive Integers as Sums of Prime Powers and Integer Powers: Analysis with Positive Density Subsets}
\author{Meng Gao\\ \small{Department of Mathematics, China University of Mining and Technology,}\\ \small{Beijing 100083, P. R. China}\\ \small{E-mail: ntmgao@outlook.com}  }

\date{}
\maketitle

\begin{abstract}
In this paper, we use the transference principle to investigate the representation of sufficiently large positive integers as the sum of prime powers and integer powers, where the primes are drawn from a positive density subset of the set of all primes , and the integer powers are drawn from a positive density subset of $ k $-th powers.
	
\begin{flushleft}
\end{flushleft}

\end{abstract}

\section{Introduction}
For each prime $ p $, define $ \tau (k,p) $  so that $p^{\tau (k,p)}\ |\ k,\ p^{\tau (k,p)+1}\ \nmid\ k  $. Let

		\begin{center}
		$\begin{aligned} R_{k}:=\prod\limits_{(p-1)|k}p^{\gamma(k,p)}\end{aligned} $,
	\end{center}
where\begin{center}
  	$ \gamma(k,p):=  
  	\begin{cases}
		 \tau(k,p)+2 \ \quad if \   p=2 \  and \   \tau(k,p)>0   ,    \\
		 \tau(k,p)+1 \ \quad otherwise.
		\end {cases}$ 
	\end{center}
Let $ \mathbb{P} $ denote the set of all primes and $ k\geq2 $ be an integer. Set $ \mathbb{N}^{(k)}:=\{t^{k}:t \in \mathbb{N}\} $ , $ \mathbb{Z}_{m}^{(k)}:=\{t^{k}:t \in \mathbb{Z}_{m}\}$ and $ Z(m):=\{a \in \mathbb{Z}_{m}^{(k)}:(a,m)=1 \} $, where $ \mathbb{Z}_{m}:=\mathbb{Z}/m\mathbb{Z}  $. For $ m \in \mathbb{N} $ let $ P(m):=\prod\limits_{p\leq m}p^{k} $ and \begin{center}
	$ \mathcal{Z}_{k}:=\lim \limits_{m\rightarrow\infty}\dfrac{|\mathbb{Z}_{P(m)}^{(k)}|}{|\{a \in \mathbb{Z}_{P(m)}^{(k)}:(a, P(m))=1\}|} $.
\end{center}
From \cite[the proof of Lemma 6.3]{Sal}, we know that\begin{center}
	$ \begin{aligned}
		\mathcal{Z}_{k}=\sum\limits_{\substack{n=1\\n \ square-free}}^{\infty}\dfrac{1}{|Z(n^{k})|}.
	\end{aligned} $
\end{center}
Let\begin{center}
	$\begin{aligned}
	\mathcal{T}_{k}:=\sum\limits_{\substack{n=1\\n \ square-free}}^{\infty}\dfrac{1}{|Z(n^{2k})|}.
	\end{aligned}  $
\end{center}
Let   $ A\subseteq \mathbb{P} $ and $ B \subseteq \mathbb{N}^{(k)} $. Define
\begin{center}
	$ \delta_{A}=\underline{\delta}_{A}(\mathbb{P}):=\liminf \limits_{N\rightarrow\infty}\dfrac{|A\bigcap[N]|}{|\mathbb{P}\bigcap[N]|} $ 
\end{center}
and
\begin{center}
	$ \delta_{B}=\underline{\delta}_{B}(\mathbb{N}^{(k)}):=\liminf \limits_{N\rightarrow\infty}\dfrac{|B\bigcap[N]|}{|\mathbb{N}^{(k)}\bigcap[N]|}  $,
\end{center}
where $ [N]:=\{1,\ldots,N\} $.
	
Let $ \mathcal{P}\subseteq  \mathbb{P} $ be such that $ \pi_{\mathcal{P}}(x) \sim \delta\pi(x) $, where $ \delta \in (0,1) $ is a constant and $ \pi_{\mathcal{P}}(x):=\#\{p \leq x: p\in \mathcal{P}\} $.

The transference principle was originally developed by Green \cite{Gre} and can be considered as a modern variant of the circle method.  The transference principle can be applied to solve additive problems in dense subsets of primes. For example, Li, Pan \cite{LP} proved that for three positive density subsets $ A_{1}, A_{2}, $ and $ A_{3} $ of $ \mathbb{P} $, if the sum of their positive lower densities is greater than 2, then for all sufficiently large positive odd integers $ n $, it can be expressed in the form $ n=p_{1}+p_{2}+p_{3} $, where $ p_{i} \in A_{i} $ for all $ i \in \{1,2,3\} $; Shao \cite{Shao} proved that if the positive lower density of a subset $ A $ in the set of prime numbers is greater than $ 5/8 $, then all sufficiently large positive odd integers can be expressed as the sum of three primes from the set $ A $. 

Motivated by the work of Li, Pan \cite{LP} and Shao \cite{Shao}, Salmensuu \cite{Sal}  proved that a positive density subset of $ k $-th powers forms an asymptotic additive basis of order $ O(k^{2}) $ provided that the relative lower density of the set is greater than $ (1-\mathcal{Z}_{k}^{-1}/2)^{1/k} $. Using the transference lemma which was established in \cite{Sal}, the author \cite{Gao} investigated the density version of Waring-Goldbach problem and proved that if $ A $ is a subset of the primes, and the lower density of $ A $ in the primes is larger than $ 1-1/2k $, then every sufficiently large natural
number $ n $ satisfying the necessary congruence condition can be expressed as a sum of $ s $ terms of the $ k $-th powers of primes from set $ A $.

In this paper, by integrating the works of Salmensuu \cite{Sal} and the author \cite{Gao}, we investigate the representation of sufficiently large positive integers as the sum of prime powers and integer powers, where the primes are drawn from a positive density subset of the set of all primes, and the integer powers are drawn from a positive density subset of $ k $-th powers. Our results are as follows.
\begin{theorem}\label{theorem 1.1}
	Let $ s_{1}, s_{2} \in \mathbb{N},\ k\in \mathbb{N}\setminus\{1,2,4,8,9\},\ s_{1}\geq 16k\omega(k)+4k+4+s_{2}  $ and $ s_{1}+s_{2}>k^{2}+k $ . Let $ \delta_{A}>1-1/2k $ and $ k\delta_{A}+\mathcal{Z}_{k}\delta_{B}^{k}>\mathcal{Z}_{k}+k-1 $. Then, for all sufficiently large integers $ n \equiv s_{1}+s_{2} \ (\bmod \ R_{k}) $, we have 
\begin{center}
	$ n=p_{1}^{k}+\cdots+p_{s_{1}}^{k}+n_{s_{1}+1}^{k}+\cdots+n_{s_{1}+s_{2}}^{k} $,
\end{center}
where $ p_{i} \in A $ for all $ i \in \{1,\ldots,s_{1}\} $ and $ n_{j}^{k} \in B $ for all $ j \in \{s_{1}+1,\ldots,s_{1}+s_{2}\} $.
\end{theorem}
\begin{theorem}\label{theorem 1.2}
	Let $ s_{1}, s_{2} \in \mathbb{N},\ k\in \mathbb{N}\setminus\{1,2,4,8,9\},\ s_{1}\geq 16k\omega(k)+4k+4+s_{2}  $ and $ s_{1}+s_{2}>k^{2}+k $ . Let $ \delta_{B}>(1-\mathcal{Z}_{k}^{-1}/2)^{1/k} $ and $ k\delta_{A}+\mathcal{Z}_{k}\delta_{B}^{k}>\mathcal{Z}_{k}+k-1 $. Then, for all sufficiently large integers $ n \equiv s_{1}+s_{2} \ (\bmod \ R_{k}) $, we have \begin{center}
		$ n=n_{1}^{k}+\cdots+n_{s_{1}}^{k}+p_{s_{1}+1}^{k}+\cdots+p_{s_{1}+s_{2}}^{k} $,
	\end{center}
	where $ n_{i}^{k} \in B $ for all $ i \in \{1,\ldots,s_{1}\} $ and $ p_{j} \in A $ for all $ j \in \{s_{1}+1,\ldots,s_{1}+s_{2}\} $.
\end{theorem}

\begin{theorem}\label{theorem 1.3}
	Let $ s_{1}, s_{2} \in \mathbb{N},\ k\in \mathbb{N}\setminus\{1,2,4,8,9\},\ s_{1}\geq 16k\omega(k)+4k+4+s_{2}  $ and $ s_{1}+s_{2}>k^{2}+k $ . Let $ \delta>1/2 $ and $ \delta+\mathcal{Z}_{k}\delta_{B}^{k}>\mathcal{Z}_{k} $. Then, for all sufficiently large integers $ n \equiv s_{1}+s_{2} \ (\bmod \ R_{k}) $, we have 
	\begin{center}
		$ n=p_{1}^{k}+\cdots+p_{s_{1}}^{k}+n_{s_{1}+1}^{k}+\cdots+n_{s_{1}+s_{2}}^{k} $,
	\end{center}
	where $ p_{i} \in \mathcal{P} $ for all $ i \in \{1,\ldots,s_{1}\} $ and $ n_{j}^{k} \in B $ for all $ j \in \{s_{1}+1,\ldots,s_{1}+s_{2}\} $.
\end{theorem}

\begin{theorem}\label{theorem 1.4}
	Let $ s_{1}, s_{2} \in \mathbb{N},\ k\in \mathbb{N}\setminus\{1,2,4,8,9\},\ s_{1}\geq 16k\omega(k)+4k+4+s_{2}  $ and $ s_{1}+s_{2}>k^{2}+k $ . Let $ \delta_{B}>(1-\mathcal{Z}_{k}^{-1}/2)^{1/k} $ and $ \delta+\mathcal{Z}_{k}\delta_{B}^{k}>\mathcal{Z}_{k} $.Then, for all sufficiently large integers $ n \equiv s_{1}+s_{2} \ (\bmod \ R_{k}) $, we have \begin{center}
		$ n=n_{1}^{k}+\cdots+n_{s_{1}}^{k}+p_{s_{1}+1}^{k}+\cdots+p_{s_{1}+s_{2}}^{k} $,
	\end{center}
	where $ n_{i}^{k} \in B $ for all $ i \in \{1,\ldots,s_{1}\} $ and $ p_{j} \in \mathcal{P} $ for all $ j \in \{s_{1}+1,\ldots,s_{1}+s_{2}\} $.
\end{theorem}

\begin{theorem}\label{theorem 1.5}
	Let $ s_{1}, s_{2} \in \mathbb{N},\ k\in \{4,8,9\},\ s_{1}\geq 16k\omega(k)+4k+4+s_{2}  $ and $ s_{1}+s_{2}>k^{2}+k $ . Let $ \delta_{A}>1-1/2k $ and $ k\delta_{A}+\mathcal{T}_{k}\delta_{B}^{k}>\mathcal{Z}_{k}+k-1 $. Then, for all sufficiently large integers $ n \equiv s_{1}+s_{2} \ (\bmod \ R_{k}) $, we have 
	\begin{center}
		$ n=p_{1}^{k}+\cdots+p_{s_{1}}^{k}+n_{s_{1}+1}^{k}+\cdots+n_{s_{1}+s_{2}}^{k} $,
	\end{center}
	where $ p_{i} \in A $ for all $ i \in \{1,\ldots,s_{1}\} $ and $ n_{j}^{k} \in B $ for all $ j \in \{s_{1}+1,\ldots,s_{1}+s_{2}\} $.
\end{theorem}
\begin{theorem}\label{theorem 1.6}
	Let $ s_{1}, s_{2} \in \mathbb{N},\ k\in \{8,9\},\ s_{1}\geq 16k\omega(k)+4k+4+s_{2}  $ and $ s_{1}+s_{2}>k^{2}+k $ . Let $ \delta_{B}>((\mathcal{Z}_{k}-1/2)\mathcal{T}_{k}^{-1})^{1/k} $ and $ k\delta_{A}+\mathcal{T}_{k}\delta_{B}^{k}>\mathcal{Z}_{k}+k-1 $.Then, for all sufficiently large integers $ n \equiv s_{1}+s_{2} \ (\bmod \ R_{k}) $, we have \begin{center}
		$ n=n_{1}^{k}+\cdots+n_{s_{1}}^{k}+p_{s_{1}+1}^{k}+\cdots+p_{s_{1}+s_{2}}^{k} $,
	\end{center}
	where $ n_{i}^{k} \in B $ for all $ i \in \{1,\ldots,s_{1}\} $ and $ p_{j} \in A $ for all $ j \in \{s_{1}+1,\ldots,s_{1}+s_{2}\} $.
\end{theorem}

\begin{theorem}\label{theorem 1.7}
	Let $ s_{1}, s_{2} \in \mathbb{N},\ k\in \{4,8,9\},\ s_{1}\geq 16k\omega(k)+4k+4+s_{2}  $ and $ s_{1}+s_{2}>k^{2}+k $ . Let $ \delta>1/2 $ and $ \delta+\mathcal{T}_{k}\delta_{B}^{k}>\mathcal{Z}_{k} $. Then, for all sufficiently large integers $ n \equiv s_{1}+s_{2} \ (\bmod \ R_{k}) $, we have 
	\begin{center}
		$ n=p_{1}^{k}+\cdots+p_{s_{1}}^{k}+n_{s_{1}+1}^{k}+\cdots+n_{s_{1}+s_{2}}^{k} $,
	\end{center}
	where $ p_{i} \in \mathcal{P} $ for all $ i \in \{1,\ldots,s_{1}\} $ and $ n_{j}^{k} \in B $ for all $ j \in \{s_{1}+1,\ldots,s_{1}+s_{2}\} $.
\end{theorem}

\begin{theorem}\label{theorem 1.8}
	Let $ s_{1}, s_{2} \in \mathbb{N},\ k\in \{8,9\},\ s_{1}\geq 16k\omega(k)+4k+4+s_{2}  $ and $ s_{1}+s_{2}>k^{2}+k $ . Let $ \delta_{B}>((\mathcal{Z}_{k}-1/2)\mathcal{T}_{k}^{-1})^{1/k} $ and $ \delta+\mathcal{T}_{k}\delta_{B}^{k}>\mathcal{Z}_{k} $.Then, for all sufficiently large integers $ n \equiv s_{1}+s_{2} \ (\bmod \ R_{k}) $, we have \begin{center}
		$ n=n_{1}^{k}+\cdots+n_{s_{1}}^{k}+p_{s_{1}+1}^{k}+\cdots+p_{s_{1}+s_{2}}^{k} $,
	\end{center}
	where $ n_{i}^{k} \in B $ for all $ i \in \{1,\ldots,s_{1}\} $ and $ p_{j} \in \mathcal{P} $ for all $ j \in \{s_{1}+1,\ldots,s_{1}+s_{2}\} $.
\end{theorem}

Similar to \cite{Sal} and \cite{Gao}, we employ the transference principle established by Salmensuu in \cite{Sal} to prove the above results. Salmensuu's transference principle is the following.
\begin{proposition}\label{pro 1.9}(\cite[Proposition 3.9]{Sal})
Let $ s\geq 2, q\in (s-1,s), $ and $ \epsilon, \eta\in (0,1) $. Let $ N $ be a sufficiently large positive integer, and let $ f_{1},\ldots,f_{s}:[N]\rightarrow \mathbb{R}_{\geq 0} $ be a set of non-negative functions defined on $ [N] $. If $ f_{1},\ldots,f_{s} $ satisfy the following conditions:

(i)(Mean condition) For any $ i\in\{1,\ldots,s\} $, we have
\begin{center}
	$ \mathbb{E}_{n\in [N]}f_{i}(n)>\epsilon/2 $,\
\end{center}
and
\begin{center}
	$ \mathbb{E}_{n\in [N]}f_{1}(n)+\cdots+f_{s}(n)>s(1+\epsilon)/2 $;\
\end{center}

(ii)(Pseudorandomness condition) For any $ i\in\{1,\ldots,s\} $, there exists a majorant function $ \nu_{i}:[N]\rightarrow \mathbb{R}_{\geq 0} $ such that $ f_{i}\leq \nu_{i} $ and
\begin{center}
	$ ||\widehat{\nu_{i}} - \widehat{1_{[N]}}||_{\infty}\leq \eta N $;\
\end{center}

(iii)(Restriction estimate) For any $ i\in\{1,\ldots,s\} $, $ ||\widehat{f_{i}}||_{q}^{q}\ll N^{q-1} $.

Then for any positive integer $ n \in \big((1-\kappa^{2})\frac{sN}{2},(1+\kappa)\frac{sN}{2}\big) $, we have
\begin{center}
	$ f_{1}\ast\cdots\ast f_{s}(n)\geq c(\epsilon,s)N^{s-1} $,
\end{center}
where $ \kappa=\epsilon/32 $.	
\end{proposition}

\begin{remark}
	In order to prove the above theorems, we choose the suitable weight characteristic functions using W-trick. When $ k\in \mathbb{N}\setminus\{1,2,4,8,9\} $, we define $ W $ by (\ref{equ1}). But when $ k\in \{4,8,9\} $ , we define $ W $ by (\ref{equ4}).  The two definitions of $ W $ in (\ref{equ1}) and (\ref{equ4}) respectively originate from \cite[Equ (8)]{Sal} and \cite[Equ (1)]{Gao}. The reason that we adopt two different definitions of $ W $ is as follows:
	
	When $ k\in \{2,4,8,9\} $, we failed to prove the pseudorandomness of (\ref{equ3})(see \cite[proof of Lemma 5.6]{Gao}). But we can prove the pseudorandomness of (\ref{equ6}) for all integers $k\geq 2 $. On the other hand, $ \mathcal{Z}_{k}>\mathcal{T}_{k} $(see Subsection 4.2). In order to obtain better density conditions, we did not choose to define $ W $ by (\ref{equ4}) when $ k\in \mathbb{N}\setminus\{1,2,4,8,9\} $.
\end{remark}
\begin{remark}
	For some small values of k, readers can refer to the table following \cite[Lemma 6.3]{Sal} for the values of $ \mathcal{Z}_{k} $. Upon calculation, we have $ \mathcal{T}_{2}\approx 2.085 $ and $ \mathcal{T}_{4}\approx 1.063 $.
	
(i)	When $ k = 2 $, we failed to draw relevant conclusions . The reason is that for any $ \delta_{A},\ \delta_{B},\ \delta\in (0,1] $, neither  \begin{center}
		$ k\delta_{A}+\mathcal{T}_{k}\delta_{B}^{k}>\mathcal{Z}_{k}+k-1  $
	\end{center}
nor \begin{center}
	$ \delta+\mathcal{T}_{k}\delta_{B}^{k}>\mathcal{Z}_{k} $
\end{center}
can hold when $ k=2 $;

(ii)In Theorem \ref{theorem 1.6} and Theorem \ref{theorem 1.8}, we did not obtain conclusions for $  k=4 $. The reason is that\begin{center}
	$ \mathcal{T}_{k}+1/2>\mathcal{Z}_{k} $
\end{center}
cannot hold when $  k=4 $.
\end{remark}

\begin{remark}
	In fact, from Theorem \ref{theorem 1.1} and Theorem \ref{theorem 1.2}, we can see that as long as $ k\delta_{A}-(k-1)+\mathcal{Z}_{k}\delta_{B}^{k}-\mathcal{Z}_{k}+1>1  $, we can express sufficiently large positive integers $ n $ satisfying the necessary  congruence condition in the form of the sum of prime powers and integer powers, where the primes are drawn from the set $ A $, and the integer powers are drawn from the set $ B $. This is because when $ k\delta_{A}-(k-1)+\mathcal{Z}_{k}\delta_{B}^{k}-\mathcal{Z}_{k}+1>1  $, we either have $ k\delta_{A}-(k-1)>1/2 $  or $ \mathcal{Z}_{k}\delta_{B}^{k}-\mathcal{Z}_{k}+1>1/2 $, which means either  $ \delta_{A}>1-1/2k  $ or $ \delta_{B}>(1-\mathcal{Z}_{k}^{-1}/2)^{1/k} $. Similar discussion is applicable to Theorems \ref{theorem 1.3} and \ref{theorem 1.4},  Theorems \ref{theorem 1.5} and \ref{theorem 1.6}, and   Theorems \ref{theorem 1.7} and \ref{theorem 1.8} .
\end{remark}

\section{Notation}
Let $ s \in \mathbb{N} $ and $ s\geq 2 $. For sets $ A,\ B \subseteq \mathbb{N} $, define the sumset by\begin{center}
	$ sA=\{a_{1}+\cdots+a_{s}:a_{1},\ldots,a_{s}\in A\},\ A+B=\{a+b:a\in A, b\in B\} $,
\end{center}
and define\begin{center}
	$ A^{(k)}=\{a^{k}:a\in A\} $.
\end{center}
For finitely supported functions $ f,g:\mathbb{Z}\rightarrow \mathbb{C} $ , we define convolution $ f\ast g  $ by 
\begin{center}
	$\begin{aligned}
	f\ast g(n)=\sum\limits_{a+b=n}f(a)g(b) \end{aligned}$.
\end{center}

The Fourier transform of a finitely supported function $ f:\mathbb{Z}\rightarrow \mathbb{C} $ is defined by
\begin{center}
	$\begin{aligned} \widehat{f}(\alpha)=\sum\limits_{n\in \mathbb{Z}}f(n)e(n \alpha) \end{aligned}$ ,
\end{center}
where $ e(x)=e^{2\pi ix} $.

For a set $ B $, we write $ 1_{B}(x) $ for  its characteristic function. If $ f:B\rightarrow \mathbb{C} $ is a function and $ B_{1} $ is a non-empty finite subset of $ B $, we write $ \mathbb{E}_{x \in B_{1}}f(x) $ for the average value of $ f $ on  $ B_{1} $,  that is to say
\begin{center}
	$\begin{aligned}
	\mathbb{E}_{x \in B_{1}}f(x)=\dfrac{1}{|B_{1}|}\sum\limits_{x \in B_{1}}f(x).
	\end{aligned}  $
\end{center}

We write $ f=o(g) $ if 
\begin{center}
	$ \lim\limits_{x\rightarrow\infty} \dfrac{f(x)}{g(x)}=0$ .
\end{center}
The function $ f $ is asymptotic to $ g $, denoted $ f\sim g $ , if
\begin{center}
	$ \lim\limits_{x\rightarrow\infty} \dfrac{f(x)}{g(x)}=1$ .
\end{center}
We will use notation $ \mathbb{T} $ for $ \mathbb{R}/\mathbb{Z} $.
We also define the $ L^{p}$-$norm $ 
\begin{center}
	$\begin{aligned}
	\|f\|_{p}=\bigg(\int_{\mathbb{T}}|f(\alpha)|^{p}d\alpha\bigg)^{1/p}\end{aligned} $
\end{center}
for a function $ f:\mathbb{T}\rightarrow \mathbb{C} $.

We write $ f \ll g $ or $ f=O(g) $ if there exists a constant $ C>0  $ such that $ |f(x)|\leq Cg(x) $ for all values of $ x $ in the domain of $ f $. The letter $ p $, with or without subscript,
denotes a prime number.

\section{ $ k\in \mathbb{N}\setminus\{1,2,4,8,9\} $ }
\subsection{Definitions}
Let $ s_{1},s_{2}\in \mathbb{N} $.
Let $ n_{0} $ be a sufficiently large positive integer satisfying $ n_{0}\equiv s_{1}+s_{2} \ (\bmod \ R_{k}) $ .
Let $ w=\log\log\log n_{0} $ and
\begin{equation}\label{equ1}
W_{1}:=\prod\limits_{1<p\leq w}p^{k} .
\end{equation}
Let $ b \in [W_{1}] $ be such that 
$ b \in Z(W_{1}) $. Define
\begin{equation}\label{equ2}
\sigma_{1}(b):=\#\{z \in [W_{1}]:z^{k}\equiv b\ (\bmod \ W_{1})\}.
\end{equation}
Let $ N:=\lfloor 2n_{0}/(sW_{1}) \rfloor $. It is not difficult to prove that
\begin{center}
	$ W_{1}=o(\log N) $.
\end{center}

Let $ A, \mathcal{P} \subseteq \mathbb{P}  $ satisfy $ \pi_{\mathcal{P}}(x)\sim \delta\pi(x),\  \delta\in (0,1]$ and $ B \subseteq \mathbb{N}^{(k)} $. Define functions $ f^{\prime}_{b},\ f^{\prime\prime}_{b},\ \mathbf{f}_{b},\ \nu^{\prime}_{b},\ \nu^{\prime\prime}_{b}:[N]\rightarrow \mathbb{R}_{\geq 0} $ by
\begin{center}
	$  f^{\prime}_{b}(n):=\begin{cases}
	\dfrac{\varphi(W_{1})}{W_{1}\sigma_{1}(b)}kp^{k-1}\log p \quad if \  W_{1}n+b=p^{k}, \ p\in A,  \\
	0  \qquad\qquad\qquad\qquad \ \ \  otherwise, 
	\end {cases}  $
\end{center}
\begin{center}
	$  f^{\prime\prime}_{b}(n):=\begin{cases}
	\dfrac{k}{\sigma_{1}(b)}t^{k-1} \quad if \  W_{1}n+b=t^{k}\in B,   \\
	0  \qquad\qquad \ \ \  otherwise, 
	\end {cases}  $
\end{center}
\begin{center}
	$  \mathbf{f}_{b}(n):=\begin{cases}
	\dfrac{\varphi(W_{1})}{W_{1}\sigma_{1}(b)}kp^{k-1}\log p \quad if \  W_{1}n+b=p^{k}, \ p\in \mathcal{P},  \\
	0  \qquad\qquad\qquad\qquad \ \ \  otherwise, 
	\end {cases}  $
\end{center}

\begin{equation}\label{equ3}
	  \nu^{\prime}_{b}(n):=\begin{cases}
	\dfrac{\varphi(W_{1})}{W_{1}\sigma_{1}(b)}kp^{k-1}\log p \quad if \  W_{1}n+b=p^{k}, \ p\in \mathbb{P},  \\
	0  \qquad\qquad\qquad\qquad \ \ \ otherwise, 
	\end {cases}  	
\end{equation}
and
\begin{center}
	$  \nu^{\prime\prime}_{b}(n):=\begin{cases}
	\dfrac{k}{\sigma_{1}(b)}t^{k-1} \quad if \  W_{1}n+b=t^{k}\in \mathbb{N}^{(k)},   \\
	0  \qquad\qquad\ \ \  otherwise. 
	\end {cases}  $
\end{center}
Define functions $ g_{1},\ g_{2},\ \mathbf{g}:[W_{1}]\times \mathbb{N}\rightarrow \mathbb{R}_{\geq 0} $ by
\begin{center}
	$ g_{1}(b,N):=\mathbb{E}_{n \in [N]}f^{\prime}_{b}(n), $
\end{center}
\begin{center}
	$ g_{2}(b,N):=\mathbb{E}_{n \in [N]}f^{\prime\prime}_{b}(n), $
\end{center}
and
\begin{center}
	$ \mathbf{g}(b,N):=\mathbb{E}_{n \in [N]}\mathbf{f}_{b}(n). $
\end{center}

It is not difficult to prove that $ \mathbb{E}_{n \in [N]} \nu^{\prime\prime}_{b}(n)\sim 1 $. Using the same method of \cite[Section 2]{Chow}, we can also prove that $ \mathbb{E}_{n \in [N]} \nu^{\prime}_{b}(n)\sim 1 $.

Let $ \epsilon \in (0,1) $. For $ b\in Z(W_{1}) $, define
\begin{center}
	$ f_{i}(b):=\max \bigg(0,\dfrac{1}{1+\epsilon}\big(g_{i}(b,N)-\epsilon/2 \big) \bigg) ,\ (i=1,2)$,
\end{center}
and
\begin{center}
	$ \mathbf{f}(b):=\max\bigg(0,\ \dfrac{1}{1+\epsilon}\big(\mathbf{g}(b,N)-\epsilon/2\big)\bigg) $.
\end{center}

Note that $ f^{\prime}_{b}(n) \leq \nu^{\prime}_{b}(n),\ f^{\prime\prime}_{b}(n) \leq \nu^{\prime\prime}_{b}(n) $ . Therefore, $ f_{i}(b)\in [0,1) $ provided that $ N $ is large enough for all $ i\in \{1,2\} $.

Unless otherwise specified, the symbols defined in this subsection will only be used in Section 3.	
\subsection{Mean value estimate}
\begin{lemma}\label{lem 3.1}
Let $ h_{1},h_{2}:Z(W_{1})\rightarrow [0,1) $ satisfy $ \mathbb{E}_{b \in Z(W_{1})}h_{1}(b)>1/2 $ and $ \mathbb{E}_{b \in Z(W_{1})}(h_{1}(b)+h_{2}(b))> 1 $. Let $ s_{2}\geq 1 $ and $ s_{1}\geq 16k\omega(k)+4k+4+s_{2} $. Then, for all $ n \in \mathbb{Z}_{W_{1}} $ with $  n \equiv s_{1}+s_{2} \ (\bmod \ R_{k}) $, there exist $ b_{1},\ldots,b_{s_{1}},b_{s_{1}+1},\ldots,b_{s_{1}+s_{2}} \in Z(W_{1}) $ such that: \\
(i) $ n\equiv b_{1}+\cdots+b_{s_{1}}+b_{s_{1}+1}+\cdots+b_{s_{1}+s_{2}} (\bmod\ W_{1})$ ;\\
(ii) $\ h_{1}(b_{i})>0 $ for all $ i \in \{1,\ldots,s_{1}\} $ and $\ h_{2}(b_{j})>0 $ for all $ j \in \{s_{1}+1,\ldots,s_{1}+s_{2}\} $;\\
(iii)
	$ h_{1}(b_{1})+\cdots+h_{1}(b_{s_{1}})+h_{2}(b_{s_{1}+1})+\cdots+h_{2}(b_{s_{1}+s_{2}})>(s_{1}+s_{2})/2. $
\end{lemma}
\begin{proof}
Let $ \mu_{i}:=\max_{b\in Z(W_{1})}h_{i}(b) $ for $ i\in \{1,2\} $ and $ \lambda=1-\mu_{1} $. Note that $ \mu_{1}+\mu_{2}\geq \mathbb{E}_{b \in Z(W_{1})}(h_{1}(b)+h_{2}(b))> 1 $. Therefore, we have $ \mu_{2}> \lambda $. Let $ A:=\{b\in Z(W_{1}):h_{1}(b)>\lambda \} $. By repeating the arguments in \cite[proof of Lemma 6.4]{Sal}, we can get \begin{center}
$  $	$ s^{\prime}A=\{a\in \mathbb Z_{W_{1}}:a\equiv s^{\prime}\ (\bmod \ R_{k}) \} $
\end{center}
for all $ s^{\prime}\geq 8k\omega(k)+2k+2 $.

Let $ b^{\prime}, b^{\prime\prime}\in Z(W_{1}) $ be such that $ h_{1}(b^{\prime})=\mu_{1}, h_{2}(b^{\prime\prime})=\mu_{2} $ and let $ s^{\prime\prime}\geq s^{\prime}+s_{2} $. By \cite[Equ (16)]{Sal}, $ s^{\prime\prime}b^{\prime}+s_{2}b^{\prime\prime}\equiv s^{\prime\prime}+s_{2}\ (\bmod \ R_{k}) $. Therefore, for each $ n\in \mathbb{Z}_{W_{1}} $ with $ n\equiv s^{\prime}+s^{\prime\prime}+s_{2}\ (\bmod \ R_{k}) $, we have $ n-s^{\prime\prime}b^{\prime}-s_{2}b^{\prime\prime}\equiv s^{\prime} \ (\bmod \ R_{k}) $. Hence, there exist $ b_{1}^{\prime},\ldots,b_{s^{\prime}}^{\prime}\in A $ such that \begin{center}
	$ n-s^{\prime\prime}b^{\prime}-s_{2}b^{\prime\prime}\equiv b_{1}^{\prime}+\cdots+b_{s^{\prime}}^{\prime}\ (\bmod \ W_{1}) $
\end{center}
and
\begin{center}
	$\begin{aligned} s^{\prime\prime}h_{1}(b^{\prime})+h_{1}(b_{1}^{\prime})+\cdots+h_{1}(b_{s^{\prime}}^{\prime})+s_{2}h_{2}(b^{\prime\prime})
	&>s^{\prime\prime}\mu_{1}+s^{\prime}\lambda+s_{2}\mu_{2}\\
	&> s^{\prime\prime}\mu_{1}+(s^{\prime}+s_{2})\lambda\\
	&=(s^{\prime\prime}-(s^{\prime}+s_{2}))\mu_{1}+(s^{\prime}+s_{2})(\lambda+\mu_{1})\\
	&> \dfrac{s^{\prime\prime}-(s^{\prime}+s_{2})}{2}+s^{\prime}+s_{2}\\
	&=\dfrac{s^{\prime\prime}+s^{\prime}+s_{2}}{2}
	 \end{aligned}$
\end{center}
\end{proof}
Next, we provide the lower bounds for  $ \mathbb{E}_{b \in Z(W_{1})}g_{1}(b,N),\ \mathbb{E}_{b \in Z(W_{1})}g_{2}(b,N) $ and $ \mathbb{E}_{b \in Z(W_{1})}\mathbf{g}(b,N) $.
\begin{lemma}\label{lem 3.2}
Let $ \epsilon \in (0,1) $. Then
\begin{center}
	$ \mathbb{E}_{b \in Z(W_{1})}g_{1}(b,N) \geq k\delta_{A}-(k-1)-\epsilon $
\end{center}
provided that $ N $ is large enough depending on $ \epsilon $.	
\end{lemma}
\begin{proof}
	See \cite[proof of Lemma 4.7]{Gao}.
\end{proof}

\begin{lemma}\label{lem 3.3}(\cite[Lemma 6.2]{Sal})
	Let $ \epsilon \in (0,1) $. Then
	\begin{center}
		$ \mathbb{E}_{b \in Z(W_{1})}g_{2}(b,N) \geq (1-\epsilon)(\mathcal{Z}_{k}\delta_{B}^{k}-\mathcal{Z}_{k}+1) $
	\end{center}
	provided that $ N $ is large enough depending on $ \epsilon $.	
\end{lemma}

\begin{lemma}\label{lem 3.4}
	Let $ \epsilon \in (0,1) $. Then
	\begin{center}
		$ \mathbb{E}_{b \in Z(W_{1})}\mathbf{g}(b,N) \geq (1-\epsilon)\delta $
	\end{center}
	provided that $ N $ is large enough depending on $ \epsilon $.	
\end{lemma}
\begin{proof}
	See \cite[proof of Lemma 4.8]{Gao}.
\end{proof}

\begin{proposition}\label{pro 3.5}
	Let $ \epsilon \in (0,1/4) $ and let $ N $ be sufficiently large depending on $ \epsilon $. Let $ \delta_{A}>1-1/2k+2\epsilon/k,\  k\delta_{A}+\mathcal{Z}_{k}\delta_{B}^{k}>\mathcal{Z}_{k}+k-1+4\epsilon $ and $ s_{2}\geq 1 ,\ s_{1}\geq 16k\omega(k)+4k+4+s_{2} $. Then, for all $ n\in \mathbb{Z}_{W_{1}} $ with $  n \equiv s_{1}+s_{2} \ (\bmod \ R_{k}) $, there exist $ b_{1},\ldots,b_{s_{1}},b_{s_{1}+1},\ldots,b_{s_{1}+s_{2}} \in Z(W_{1}) $ such that: \\
	(i) $ n\equiv b_{1}+\cdots+b_{s_{1}}+b_{s_{1}+1}+\cdots+b_{s_{1}+s_{2}} (\bmod\ W_{1})$ ;\\
	(ii) $\ g_{1}(b_{i},N)>\epsilon /2 $ for all $ i \in \{1,\ldots,s_{1}\} $ and $\ g_{2}(b_{j},N)>\epsilon/2 $ for all $ j \in \{s_{1}+1,\ldots,s_{1}+s_{2}\} $;\\
	(iii)
	$ g_{1}(b_{1},N)+\cdots+g_{1}(b_{s_{1}},N)+g_{2}(b_{s_{1}+1},N)+\cdots+g_{2}(b_{s_{1}+s_{2}},N)>(s_{1}+s_{2})(1+\epsilon)/2. $
\end{proposition}
\begin{proof}
	Note that $ \delta_{A}>1-1/2k+2\epsilon/k $ is equivalent to $ k\delta_{A}-(k-1)-\epsilon>1/2+\epsilon $.
By Lemma \ref{lem 3.2} , we have \begin{center}
	$ \mathbb{E}_{b \in Z(W_{1})}f_{1}(b)\geq \dfrac{1}{1+\epsilon}\mathbb{E}_{b \in Z(W_{1})}\big(g_{1}(b,N)-\epsilon/2\big)>1/2 $.
\end{center}  	
Note that $ k\delta_{A}+\mathcal{Z}_{k}\delta_{B}^{k}>\mathcal{Z}_{k}+k-1+4\epsilon $ is equivalent to $ k\delta_{A}-(k-1)-\epsilon+\mathcal{Z}_{k}\delta_{B}^{k}-\mathcal{Z}_{k}+1-\epsilon>1+2\epsilon $. Therefore, by Lemma \ref{lem 3.2} and Lemma \ref{lem 3.3}, we have
\begin{center}
	$\begin{aligned} \mathbb{E}_{b \in Z(W_{1})}(f_{1}(b)+f_{2}(b))
	&\geq \dfrac{1}{1+\epsilon}\mathbb{E}_{b \in Z(W_{1})}\big(g_{1}(b,N)+g_{2}(b,N)-\epsilon\big)\\
	&\geq \dfrac{(k\delta_{A}-(k-1)-\epsilon)+(\mathcal{Z}_{k}\delta_{B}^{k}-\mathcal{Z}_{k}+1-\epsilon)-\epsilon}{1+\epsilon}\\
	&>1.
	\end{aligned} $
\end{center}	
Therefore, by Lemma \ref{lem 3.1}, for all $ n \in \mathbb{Z}_{W_{1}} $ with $  n \equiv s_{1}+s_{2} \ (\bmod \ R_{k}) $, there exist $ b_{1},\ldots,b_{s_{1}},b_{s_{1}+1},\ldots,b_{s_{1}+s_{2}} \in Z(W_{1}) $ such that: \\
(i) $ n\equiv b_{1}+\cdots+b_{s_{1}}+b_{s_{1}+1}+\cdots+b_{s_{1}+s_{2}} (\bmod\ W)$ ;\\
(ii) $\ f_{1}(b_{i})>0 $ for all $ i \in \{1,\ldots,s_{1}\} $ and $\ f_{2}(b_{j})>0 $ for all $ j \in \{s_{1}+1,\ldots,s_{1}+s_{2}\} $;\\
(iii)
$ f_{1}(b_{1})+\cdots+f_{1}(b_{s_{1}})+f_{2}(b_{s_{1}+1})+\cdots+f_{2}(b_{s_{1}+s_{2}})>(s_{1}+s_{2})/2. $	

By definitions of $ f_{1} $ and $ f_{2} $, we have $\ g_{1}(b_{i},N)>\epsilon /2 $ for all $ i \in \{1,\ldots,s_{1}\} $ and $\ g_{2}(b_{j},N)>\epsilon/2 $ for all $ j \in \{s_{1}+1,\ldots,s_{1}+s_{2}\} $. We also have\begin{center}
	$\begin{aligned}
	 g_{1}(b_{1},N)+\cdots+g_{1}(b_{s_{1}},N)+g_{2}(b_{s_{1}+1},N)+\cdots+g_{2}(b_{s_{1}+s_{2}},N)&>\dfrac{(1+\epsilon)(s_{1}+s_{2})}{2}+\dfrac{\epsilon(s_{1}+s_{2})}{2}\\
	 &>\dfrac{(1+\epsilon)(s_{1}+s_{2})}{2}. \end{aligned}$
\end{center}	
\end{proof}
\begin{proposition}\label{pro 3.6}
	Let $ \epsilon \in (0,1/6) $ and let $ N $ be sufficiently large depending on $ \epsilon $. Let $ \delta_{B}>(1-(1/2-3\epsilon)\mathcal{Z}_{k}^{-1})^{1/k},\  k\delta_{A}+\mathcal{Z}_{k}\delta_{B}^{k}>\mathcal{Z}_{k}+k-1+4\epsilon $ and $ s_{2}\geq 1 ,\ s_{1}\geq 16k\omega(k)+4k+4+s_{2} $. Then, for all $ n\in \mathbb{Z}_{W_{1}} $ with $  n \equiv s_{1}+s_{2} \ (\bmod \ R_{k}) $, there exist $ b_{1},\ldots,b_{s_{1}},b_{s_{1}+1},\ldots,b_{s_{1}+s_{2}} \in Z(W_{1}) $ such that: \\
	(i) $ n\equiv b_{1}+\cdots+b_{s_{1}}+b_{s_{1}+1}+\cdots+b_{s_{1}+s_{2}} (\bmod\ W_{1})$ ;\\
	(ii) $\ g_{2}(b_{i},N)>\epsilon /2 $ for all $ i \in \{1,\ldots,s_{1}\} $ and $\ g_{1}(b_{j},N)>\epsilon/2 $ for all $ j \in \{s_{1}+1,\ldots,s_{1}+s_{2}\} $;\\
	(iii)
	$ g_{2}(b_{1},N)+\cdots+g_{2}(b_{s_{1}},N)+g_{1}(b_{s_{1}+1},N)+\cdots+g_{1}(b_{s_{1}+s_{2}},N)>(s_{1}+s_{2})(1+\epsilon)/2. $
\end{proposition}	
\begin{proof}
Note that $ \delta_{B}>(1-(1/2-3\epsilon)\mathcal{Z}_{k}^{-1})^{1/k} $ is equivalent to $ \mathcal{Z}_{k}\delta_{B}^{k}-\mathcal{Z}_{k}+1>1/2+3\epsilon $. Therefore, by Lemma \ref{lem 3.3}, we have \begin{center}
	$ \mathbb{E}_{b \in Z(W_{1})}g_{2}(b,N) > (1-\epsilon)(1/2+3\epsilon)>1/2+2\epsilon $
\end{center}
provided that $ N $	is large enough depending on $ \epsilon $.

Therefore, \begin{center}
	$ \mathbb{E}_{b \in Z(W_{1})}f_{2}(b)\geq \dfrac{1}{1+\epsilon}\mathbb{E}_{b \in Z(W_{1})}\big(g_{2}(b,N)-\epsilon/2\big)>1/2 $.
\end{center}
Similarly, using another density condition, we have $ \mathbb{E}_{b \in Z(W_{1})}(f_{1}(b)+f_{2}(b))>1 $. By Lemma \ref{lem 3.1}, repeating the similar arguments in the proof of Lemma \ref{pro 3.5}, we can get the conclusions.  
\end{proof}

\begin{proposition}\label{pro 3.7}
	Let $ \epsilon \in (0,1/6) $ and let $ N $ be sufficiently large depending on $ \epsilon $. Let $ \delta>1/2+3\epsilon,\  \delta+\mathcal{Z}_{k}\delta_{B}^{k}>\mathcal{Z}_{k}+4\epsilon $ and $ s_{2}\geq 1 ,\ s_{1}\geq 16k\omega(k)+4k+4+s_{2} $. Then, for all $ n\in \mathbb{Z}_{W_{1}} $ with $  n \equiv s_{1}+s_{2} \ (\bmod \ R_{k}) $, there exist $ b_{1},\ldots,b_{s_{1}},b_{s_{1}+1},\ldots,b_{s_{1}+s_{2}} \in Z(W_{1}) $ such that: \\
	(i) $ n\equiv b_{1}+\cdots+b_{s_{1}}+b_{s_{1}+1}+\cdots+b_{s_{1}+s_{2}} (\bmod\ W_{1})$ ;\\
	(ii) $\ \mathbf{g}(b_{i},N)>\epsilon /2 $ for all $ i \in \{1,\ldots,s_{1}\} $ and $\ g_{2}(b_{j},N)>\epsilon/2 $ for all $ j \in \{s_{1}+1,\ldots,s_{1}+s_{2}\} $;\\
	(iii)
	$ \mathbf{g}(b_{1},N)+\cdots+\mathbf{g}(b_{s_{1}},N)+g_{2}(b_{s_{1}+1},N)+\cdots+g_{2}(b_{s_{1}+s_{2}},N)>(s_{1}+s_{2})(1+\epsilon)/2. $
\end{proposition}	
\begin{proof}
By Lemma \ref{lem 3.4}, \begin{center}
	$ \mathbb{E}_{b \in Z(W_{1})}\mathbf{f}(b)\geq \dfrac{1}{1+\epsilon}\mathbb{E}_{b \in Z(W_{1})}\big(\mathbf{g}(b,N)-\epsilon/2\big)>1/2 $.
\end{center}
Note that $ \delta+\mathcal{Z}_{k}\delta_{B}^{k}>\mathcal{Z}_{k}+4\epsilon $
is equivalent to $ \delta+\mathcal{Z}_{k}\delta_{B}^{k}-\mathcal{Z}_{k}+1>1+4\epsilon $. Therefore, by Lemma \ref{lem 3.3} and Lemma \ref{lem 3.4}, we have \begin{center}
	$\begin{aligned} \mathbb{E}_{b \in Z(W_{1})}(\mathbf{f}(b)+f_{2}(b))
	&\geq \dfrac{1}{1+\epsilon}\mathbb{E}_{b \in Z(W_{1})}\big(\mathbf{g}(b,N)+g_{2}(b,N)-\epsilon\big)\\
	&\geq \dfrac{(1-\epsilon)(\delta+\mathcal{Z}_{k}\delta_{B}^{k}-\mathcal{Z}_{k}+1)-\epsilon}{1+\epsilon}\\
	&>1.
	\end{aligned} $
\end{center}
By Lemma \ref{lem 3.1}, repeating the similar arguments in the proof of Proposition \ref{pro 3.5}, we can get the conclusions. 

\end{proof}
\begin{proposition}\label{pro 3.8}
	Let $ \epsilon \in (0,1/6) $ and let $ N $ be sufficiently large depending on $ \epsilon $. Let $ \delta_{B}>(1-(1/2-3\epsilon)\mathcal{Z}_{k}^{-1})^{1/k},\  \delta+\mathcal{Z}_{k}\delta_{B}^{k}>\mathcal{Z}_{k}+4\epsilon  $ and $ s_{2}\geq 1 ,\ s_{1}\geq 16k\omega(k)+4k+4+s_{2} $. Then, for all $ n\in \mathbb{Z}_{W_{1}} $ with $  n \equiv s_{1}+s_{2} \ (\bmod \ R_{k}) $, there exist $ b_{1},\ldots,b_{s_{1}},b_{s_{1}+1},\ldots,b_{s_{1}+s_{2}} \in Z(W_{1}) $ such that: \\
	(i) $ n\equiv b_{1}+\cdots+b_{s_{1}}+b_{s_{1}+1}+\cdots+b_{s_{1}+s_{2}} (\bmod\ W_{1})$ ;\\
	(ii) $\ g_{2}(b_{i},N)>\epsilon /2 $ for all $ i \in \{1,\ldots,s_{1}\} $ and $\ \mathbf{g}(b_{j},N)>\epsilon/2 $ for all $ j \in \{s_{1}+1,\ldots,s_{1}+s_{2}\} $;\\
	(iii)
	$ g_{2}(b_{1},N)+\cdots+g_{2}(b_{s_{1}},N)+\mathbf{g}(b_{s_{1}+1},N)+\cdots+\mathbf{g}(b_{s_{1}+s_{2}},N)>(s_{1}+s_{2})(1+\epsilon)/2. $
\end{proposition}	
\begin{proof}
	
Note that $ \delta_{B}>(1-(1/2-3\epsilon)\mathcal{Z}_{k}^{-1})^{1/k} $ is equivalent to $ \mathcal{Z}_{k}\delta_{B}^{k}-\mathcal{Z}_{k}+1>1/2+3\epsilon $. Therefore, by Lemma \ref{lem 3.3}, we have \begin{center}
	$ \mathbb{E}_{b \in Z(W_{1})}g_{2}(b,N) > (1-\epsilon)(1/2+3\epsilon)>1/2+2\epsilon $
\end{center}
provided that $ N $	is large enough depending on $ \epsilon $.

Therefore, \begin{center}
	$ \mathbb{E}_{b \in Z(W_{1})}f_{2}(b)\geq \dfrac{1}{1+\epsilon}\mathbb{E}_{b \in Z(W_{1})}\big(g_{2}(b,N)-\epsilon/2\big)>1/2 $.
\end{center}
Similarly, using another density condition, we have $ \mathbb{E}_{b \in Z(W)}(\mathbf{f}(b)+f_{2}(b))>1 $. By Lemma \ref{lem 3.1}, repeating the similar arguments in the proof of Proposition \ref{pro 3.5}, we can get the conclusions.

\end{proof}
\subsection{Conclusions}
\begin{proposition}\label{pro 3.9}
	Let $ \alpha \in \mathbb{T} $ . For $ b\in [W_{1}] $ with $ b \in Z(W_{1}) $, we have
	\begin{center}
		$ \mid \widehat{\nu_{b}^{\prime}}(\alpha) - \widehat{1_{[N]}}(\alpha) \mid =o(N). $
	\end{center}
\end{proposition}
\begin{proof}
	See \cite[Section 5]{Gao}.
\end{proof}
\begin{proposition}\label{pro 3.10}(\cite[Proposition 4.2]{Sal})
	Let $ \alpha \in \mathbb{T} $ . For $ b\in [W_{1}] $ with $ b \in Z(W_{1}) $, we have
	\begin{center}
		$ \mid \widehat{\nu_{b}^{\prime\prime}}(\alpha) - \widehat{1_{[N]}}(\alpha) \mid =o(N). $
	\end{center}
\end{proposition}
\begin{proposition}\label{pro 3.11}
	Let $ s\in \mathbb{N} $ and $ s>k(k+1) $. For $ b\in [W_{1}] $ with $ b\in Z(W_{1}) $ and $ q\in (s-1,s) $ , we have \begin{center}
		$ \| \widehat{f_{b}^{\prime}} \|_{q} \ll N^{1-1/q} $.
	\end{center}	
\end{proposition}
\begin{proof}
	See \cite[proof of Proposition 6.3]{Gao}.
\end{proof}
\begin{proposition}\label{pro 3.12}(\cite[Proposition 4.3]{Sal})
	Let $ s\in \mathbb{N} $ and $ s>k(k+1) $. For $ b\in [W_{1}] $ with $ b\in Z(W_{1}) $ , there exists $ q\in (s-1,s) $ such that \begin{center}
	$ \| \widehat{f_{b}^{\prime\prime}} \|_{q} \ll N^{1-1/q} $.
\end{center}	
\end{proposition}
\begin{proposition}\label{pro 3.13}
	Let $ s>k(k+1) $. For $ b\in [W_{1}] $ with $ b\in Z(W_{1}) $ and  $ q\in (s-1,s) $ , we have \begin{center}
		$ \| \widehat{\mathbf{f}}_{b} \|_{q} \ll N^{1-1/q} $.
	\end{center}	
\end{proposition}
\begin{proof}
	See \cite[proof of Proposition 6.4]{Gao}.
\end{proof}

In \cite{Gao}, the author use Chow's method \cite[Section 5]{Chow} to prove \cite[Proposition 6.3]{Gao} and \cite[Proposition 6.4]{Gao}. In fact, using Chow's method, we can prove Proposition \ref{pro 3.11} and Proposition \ref{pro 3.13} for any $ q\in (s-1,s) $.

$ \mathit{Proof \ of \ Theorem \ \ref{theorem 1.1}} $. Recall the definition of $ n_{0} $ at the beginning of Subsection 3.1. Our aim is to prove that $ n_{0} \in s_{1}A^{(k)}+s_{2}B  $ provided that $ n_{0} $ is sufficiently large.

Choose $ \epsilon \in (0,1/4) $ such that $ \delta_{A}>1-1/2k+2\epsilon/k $ and $ k\delta_{A}+\mathcal{Z}_{k}\delta_{B}^{k}>\mathcal{Z}_{k}+k-1+4\epsilon $ . By Proposition \ref{pro 3.5}, there exist $ b_{1},\ldots,b_{s_{1}},b_{s_{1}+1},\ldots,b_{s_{1}+s_{2}} \in [W_{1}] $ such that $ n_{0}\equiv b_{1}+\cdots+b_{s_{1}}+b_{s_{1}+1}+\cdots+b_{s_{1}+s_{2}} (\bmod\ W_{1})$ , $ (b_{i} \bmod W_{1})\in Z(W_{1}) $, for all $ i \in \{1,\ldots,s_{1}+s_{2}\} $, and the functions $ f^{\prime}_{b_{1}},\ldots,f^{\prime}_{b_{s_{1}}},f^{\prime\prime}_{b_{s_{1}+1}},\ldots,f^{\prime\prime}_{b_{s_{1}+s_{2}}} $ satisfy the mean conditions 
\begin{center}
	$ \mathbb{E}_{n\in [N]}f^{\prime}_{b_{1}}(n)+\cdots+f^{\prime}_{b_{s_{1}}}(n)+f^{\prime\prime}_{b_{s_{1}+1}}(n)+\cdots+f^{\prime\prime}_{b_{s_{1}+s_{2}}}(n)>\dfrac{(1+\epsilon)(s_{1}+s_{2})}{2} $
\end{center}
and\begin{center}
	$ \mathbb{E}_{n\in [N]}f^{\prime}_{b_{i}}(n)>\epsilon/2,\ \mathbb{E}_{n\in [N]}f^{\prime\prime}_{b_{j}}(n)>\epsilon/2 $
\end{center}
for all $ i \in \{1,\ldots,s_{1}\},\ j \in \{s_{1}+1,\ldots,s_{1}+s_{2}\} $ . 

By Propositions \ref{pro 3.9}-\ref{pro 3.12}, pseudorandomness condition and restriction condition of Proposition \ref{pro 1.9} hold for the functions $ f^{\prime}_{b_{1}},\ldots,f^{\prime}_{b_{s_{1}}},f^{\prime\prime}_{b_{s_{1}+1}},\ldots,f^{\prime\prime}_{b_{s_{1}+s_{2}}} $ for some $ q\in (s_{1}+s_{2}-1,s_{1}+s_{2})$ and for any $ \eta>0 $.
Let $ N $ be sufficiently large depending on $ \epsilon $ and $ \eta $ be sufficiently small . By Proposition \ref{pro 1.9},\begin{center}
	$ f^{\prime}_{b_{1}}\ast\cdots\ast f^{\prime}_{b_{s_{1}}}\ast f^{\prime\prime}_{b_{s_{1}+1}}\ast\cdots\ast f^{\prime\prime}_{b_{s_{1}+s_{2}}}(n)>0 $
\end{center}
for all $ n\in \bigg(\dfrac{1-\kappa^{2}}{2}sN,\dfrac{1+\kappa}{2}sN\bigg) $, where $ \kappa=\dfrac{\epsilon}{32} $. Therefore, for all such $ n $,\begin{center}
	$ W_{1}n+b_{1}+\cdots+b_{s_{1}+s_{2}}\in s_{1}A^{(k)}+s_{2}B  $.
\end{center}
Let $  n=\dfrac{n_{0}-b_{1}-\cdots-b_{s_{1}+s_{2}}}{W_{1}}\in \mathbb{N}$. Clearly $ n\sim \dfrac{sN}{2} $.  Therefore, we have $ n\in \bigg(\dfrac{1-\kappa^{2}}{2}sN,\dfrac{1+\kappa}{2}sN\bigg) $ when $ N $ is sufficiently large depending on $ \kappa $. As a result, $ n_{0} \in s_{1}A^{(k)}+s_{2}B  $.\qquad\qquad \qquad \qquad \qquad \qquad \qquad \qquad \quad   $ \qedsymbol $

$ \mathit{Proof \ of \ Theorem \ \ref{theorem 1.2}} $. Recall the definition of $ n_{0} $ at the beginning of Subsection 3.1. Our aim is to prove that $ n_{0} \in s_{1}B+s_{2}A^{(k)}  $ provided that $ n_{0} $ is sufficiently large.

Choose $ \epsilon \in (0,1/6) $ such that $  \delta_{B}>(1-(1/2-3\epsilon)\mathcal{Z}_{k}^{-1})^{1/k}$ and $ k\delta_{A}+\mathcal{Z}_{k}\delta_{B}^{k}>\mathcal{Z}_{k}+k-1+4\epsilon   $.  The remaining proof is omitted, since it follows directly by repeating the arguments in the proof of Theorem \ref{theorem 1.1} with Proposition \ref{pro 3.6} in place of Proposition \ref{pro 3.5}.\qquad\qquad\qquad\qquad\qquad\qquad \qquad \qquad \qquad \qquad \qquad \qquad \quad \ \ $ \qedsymbol $

$ \mathit{Proof \ of \ Theorem \ \ref{theorem 1.3}} $. Recall the definition of $ n_{0} $ at the beginning of Subsection 3.1. Our aim is to prove that $ n_{0} \in s_{1}\mathcal{P}^{(k)}+s_{2}B  $ provided that $ n_{0} $ is sufficiently large.

Choose $ \epsilon \in (0,1/6) $ such that $  \delta>1/2+3\epsilon$ and $ \delta+\mathcal{Z}_{k}\delta_{B}^{k}>\mathcal{Z}_{k}+4\epsilon  $.  The remaining proof is omitted, since it follows directly by repeating the arguments in the proof of Theorem \ref{theorem 1.1} with Proposition \ref{pro 3.7} and Proposition \ref{pro 3.13} in place of Proposition \ref{pro 3.5} and Proposition \ref{pro 3.11} respectively.\qquad\qquad \qquad \qquad \qquad \qquad \qquad \qquad \quad \ \ $ \qedsymbol $

$ \mathit{Proof \ of \ Theorem \ \ref{theorem 1.4}} $. Recall the definition of $ n_{0} $ at the beginning of Subsection 3.1. Our aim is to prove that $ n_{0} \in s_{1}B+s_{2}\mathcal{P}^{(k)}  $ provided that $ n_{0} $ is sufficiently large.

Choose $ \epsilon \in (0,1/6) $ such that $  \delta_{B}>(1-(1/2-3\epsilon)\mathcal{Z}_{k}^{-1})^{1/k}$ and $ \delta+\mathcal{Z}_{k}\delta_{B}^{k}>\mathcal{Z}_{k}+4\epsilon   $.  The remaining proof is omitted, since it follows directly by repeating the arguments in the proof of Theorem \ref{theorem 1.1} with Proposition \ref{pro 3.8} and Proposition \ref{pro 3.13} in place of Proposition \ref{pro 3.5} and Proposition \ref{pro 3.11} respectively.\qquad \qquad \qquad \qquad \quad \ \ $ \qedsymbol $

\section{$ k\in \{4,8,9\} $ }
\subsection{Definitions}
Let $ s_{1},s_{2}\in \mathbb{N} $.
Let $ n_{0} $ be a sufficiently large positive integer satisfying $ n_{0}\equiv s \ (\bmod \ R_{k}) $ .
Let $ w=\log\log\log n_{0} $, \begin{center}
	$ W_{1}:=\prod\limits_{1<p\leq w}p^{k} $
\end{center}  and
\begin{equation}\label{equ4}
W_{2}:=\prod\limits_{1<p\leq w}p^{2k} .
\end{equation}
Let $ b \in [W_{2}] $ be such that 
$ b \in Z(W_{2}) $. Define
\begin{equation}\label{equ5}
\sigma_{2}(b):=\#\{z \in [W_{2}]:z^{k}\equiv b(\bmod \ W_{2})\}.
\end{equation}
Let $ N:=\lfloor 2n_{0}/(sW_{2}) \rfloor $. It is not difficult to prove that
\begin{center}
	$ W_{2}=o(\log N) $.
\end{center}

Let $ A, \mathcal{P} \subseteq \mathbb{P}  $ satisfy $ \pi_{\mathcal{P}}(x)\sim \delta\pi(x),\  \delta\in (0,1]$ and $ B \in \mathbb{N}^{(k)} $. Define functions $ f^{\prime}_{b},\ f^{\prime\prime}_{b},\ \mathbf{f}_{b},\ \nu^{\prime}_{b},\ \nu^{\prime\prime}_{b}:[N]\rightarrow \mathbb{R}_{\geq 0} $ by
\begin{center}
	$  f^{\prime}_{b}(n):=\begin{cases}
	\dfrac{\varphi(W_{2})}{W_{2}\sigma_{2}(b)}kp^{k-1}\log p \quad if \  W_{2}n+b=p^{k}, \ p\in A,  \\
	0  \qquad\qquad\qquad\qquad \  otherwise, 
	\end {cases}  $
\end{center}
\begin{center}
	$  f^{\prime\prime}_{b}(n):=\begin{cases}
	\dfrac{k}{\sigma_{2}(b)}t^{k-1} \quad if \  W_{2}n+b=t^{k}\in B,   \\
	0  \qquad\qquad\qquad\qquad \  otherwise, 
	\end {cases}  $
\end{center}
\begin{center}
	$  \mathbf{f}_{b}(n):=\begin{cases}
	\dfrac{\varphi(W_{2})}{W_{2}\sigma_{2}(b)}kp^{k-1}\log p \quad if \  W_{2}n+b=p^{k}, \ p\in \mathcal{P},  \\
	0  \qquad\qquad\qquad\qquad \  otherwise, 
	\end {cases}  $
\end{center}

\begin{equation}\label{equ6}
	  \nu^{\prime}_{b}(n):=\begin{cases}
	\dfrac{\varphi(W_{2})}{W_{2}\sigma_{2}(b)}kp^{k-1}\log p \quad if \  W_{2}n+b=p^{k}, \ p\in \mathbb{P},  \\
	0  \qquad\qquad\qquad\qquad \ otherwise,
	\end {cases}  	
\end{equation}
and
\begin{center}
	$  \nu^{\prime\prime}_{b}(n):=\begin{cases}
	\dfrac{k}{\sigma_{2}(b)}t^{k-1} \quad if \  W_{2}n+b=t^{k}\in \mathbb{N}^{(k)},   \\
	0  \qquad\qquad\qquad\qquad \  otherwise.
	\end {cases}  $
\end{center}
Define functions $ g_{1},\ g_{2},\ \mathbf{g}:[W_{2}]\times \mathbb{N}\rightarrow \mathbb{R}_{\geq 0} $ by
\begin{center}
	$ g_{1}(b,N):=\mathbb{E}_{n \in [N]}f^{\prime}_{b}(n), $
\end{center}
\begin{center}
	$ g_{2}(b,N):=\mathbb{E}_{n \in [N]}f^{\prime\prime}_{b}(n), $
\end{center}
and
\begin{center}
	$ \mathbf{g}(b,N):=\mathbb{E}_{n \in [N]}\mathbf{f}_{b}(n). $
\end{center}

It is not difficult to prove that $ \mathbb{E}_{n \in [N]} \nu^{\prime\prime}_{b}(n)\sim 1 $. Using the same method of \cite[Section 2]{Chow}, we can also prove that $ \mathbb{E}_{n \in [N]} \nu^{\prime}_{b}(n)\sim 1 $.

Unless otherwise specified, the symbols defined in this subsection will only be used in Section 4.
\subsection{Mean value estimate}
\begin{lemma}\label{lem 4.1}
	Let $ h_{1},h_{2}:Z(W_{2})\rightarrow [0,1) $ satisfy $ \mathbb{E}_{b \in Z(W_{2})}h_{1}(b)>1/2 $ and $ \mathbb{E}_{b \in Z(W_{2})}(h_{1}(b)+h_{2}(b))> 1 $. Let $ s_{2}\geq 1 $ and $ s_{1}\geq 16k\omega(k)+4k+4+s_{2} $. Then, for all $ n \in \mathbb{Z}_{W_{2}} $ with $  n \equiv s_{1}+s_{2} \ (\bmod \ R_{k}) $, there exist $ b_{1},\ldots,b_{s_{1}},b_{s_{1}+1},\ldots,b_{s_{1}+s_{2}} \in Z(W_{2}) $ such that: \\
	(i) $ n\equiv b_{1}+\cdots+b_{s_{1}}+b_{s_{1}+1}+\cdots+b_{s_{1}+s_{2}} (\bmod\ W_{2})$ ;\\
	(ii) $\ h_{1}(b_{i})>0 $ for all $ i \in \{1,\ldots,s_{1}\} $ and $\ h_{2}(b_{j})>0 $ for all $ j \in \{s_{1}+1,\ldots,s_{1}+s_{2}\} $;\\
	(iii)
	$ h_{1}(b_{1})+\cdots+h_{1}(b_{s_{1}})+h_{2}(b_{s_{1}+1})+\cdots+h_{2}(b_{s_{1}+s_{2}})>(s_{1}+s_{2})/2. $
\end{lemma}

\begin{proof}
	Let $ \mu_{i}:=\max_{b\in Z(W_{2})}h_{i}(b) $ for $ i\in \{1,2\} $ and $ \lambda=1-\mu_{1} $. Note that $ \mu_{1}+\mu_{2}\geq \mathbb{E}_{b \in Z(W_{2})}(h_{1}(b)+h_{2}(b))> 1 $. Therefore, we have $ \mu_{2}> \lambda $. Let $ A:=\{b\in Z(W_{2}):h_{1}(b)>\lambda \} $. By repeating the arguments in \cite[proof of Lemma 6.4]{Sal} with \cite[Proposition 4.2]{Gao} in place of \cite[Proposition 5.2]{Sal}, we can get \begin{center}
		$   s^{\prime}A=\{a\in \mathbb Z_{W_{2}}:a\equiv s^{\prime}\ (\bmod \ R_{k}) \} $
	\end{center}
	for all $ s^{\prime}\geq 8k\omega(k)+2k+2 $.

The remaining proof is omitted, since it follows directly by repeating the arguments in the proof of Lemma \ref{lem 3.1}.	
	
\end{proof}

Recall that\begin{center}
	$\begin{aligned}
	\mathcal{T}_{k}:=\sum\limits_{\substack{n=1\\n \ square-free}}^{\infty}\dfrac{1}{|Z(n^{2k})|}.
	\end{aligned}  $
\end{center}
Clearly, $ \mathcal{T}_{k} $ is convergent and $ \mathcal{T}_{k}\geq 1 $. On the other hand, $ \mathcal{T}_{k} < \mathcal{Z}_{k} $. The reasoning is as follows:

Let $ p $ be an odd prime. Then\begin{center}
	$ |Z(p^{2k})|=\dfrac{p^{2k-1}(p-1)}{(k,p^{2k-1}(p-1))}> \dfrac{p^{k-1}(p-1)}{(k,p^{k-1}(p-1))}=|Z(p^{k})| $.
\end{center}
In addition, we have
\begin{center}
	$ |Z(2^{2k})|=\dfrac{2}{(k,2)}\cdot\dfrac{2^{2k-2}}{(k,2^{2k-2})}>
	\dfrac{2}{(k,2)}\cdot\dfrac{2^{k-2}}{(k,2^{k-2})}=|Z(2^{k})| $.
\end{center}
By \cite[Equ (14)]{Sal}, for each square-free number $ n $, we have \begin{center}
	$ |Z(n^{2k})|> |Z(n^{k})| $.
\end{center}
From \cite[the proof of Lemma 6.3]{Sal}, we know that
\begin{center}
	$\begin{aligned}
	\mathcal{Z}_{k}=\sum\limits_{\substack{n=1\\n \ square-free}}^{\infty}\dfrac{1}{|Z(n^{k})|}.
	\end{aligned}  $ 
\end{center}

Therefore, $ \mathcal{T}_{k}< \mathcal{Z}_{k} $. In \cite{Sal}, Salmensuu has proved that $ \lim_{k \rightarrow \infty}\mathcal{Z}_{k}=1 $. Hence, we have $ \lim_{k \rightarrow \infty}\mathcal{T}_{k}=1 $.

Next, we establish the upper and lower bounds of $ \frac{|\mathbb{Z}_{W_{2}}^{(k)}|}{|Z(W_{2})|} $.
\begin{lemma}\label{lem 4.2}
Let $ \epsilon\in (0,1)$. Then\begin{center}
	$ (1-\epsilon)\mathcal{T}_{k}\leq \dfrac{|\mathbb{Z}_{W_{2}}^{(k)}|}{|Z(W_{2})|}\leq (1+\epsilon)\mathcal{Z}_{k} $
\end{center}
provided that $ n_{0} $	is sufficiently large.
\end{lemma}
\begin{proof}
	Note that
	\begin{center}
		$\begin{aligned}
	\dfrac{|\mathbb{Z}_{W_{2}}^{(k)}|}{|Z(W_{2})|}&=\prod\limits_{2\leq p\leq \omega}\dfrac{|\mathbb{Z}_{p^{2k}}^{(k)}|}{|Z(p^{2k})|}\\
	&=\dfrac{1+\frac{2^{2k-1}}{(k,2)}\big(\frac{1}{(k,2^{2k-2})}+\frac{1}{2^{k}(k,2^{k-2})}\big)}{2^{2k-1}/(k,2)(k,2^{2k-2})}\prod\limits_{2<p\leq \omega}\dfrac{1+\frac{p^{2k-1}(p-1)}{(k,p-1)}\big(\frac{1}{(k,p^{2k-1})}+\frac{1}{p^{k}(k,p^{k-1})}\big)}{p^{2k-1}(p-1)/(k,p-1)(k,p^{2k-1})}\\
	&>\dfrac{1+\frac{2^{2k-1}}{(k,2)(k,2^{2k-2})}}{2^{2k-1}/(k,2)(k,2^{2k-2})}\prod\limits_{2<p\leq \omega}\dfrac{1+\frac{p^{2k-1}(p-1)}{(k,p-1)(k,p^{2k-1})}}{p^{2k-1}(p-1)/(k,p-1)(k,p^{2k-1})}\\
	&=\dfrac{1+|Z(2^{2k})|}{|Z(2^{2k})|}\prod\limits_{2<p\leq \omega}\dfrac{1+|Z(p^{2k})|}{|Z(p^{2k})|}\\
	&=\prod\limits_{2\leq p\leq \omega}\dfrac{1+|Z(p^{2k})|}{|Z(p^{2k})|}\\
	&=\prod\limits_{p|W_{2}}\big(1+1/|Z(p^{2k})|\big)\\
	&=\sum\limits_{d|W_{2}^{1/2k}}\dfrac{1}{Z(d^{2k})}.
		\end{aligned}  $
	\end{center}
and\begin{center}
	$ \begin{aligned}
	\lim\limits_{\omega\rightarrow\infty}\sum\limits_{d|W_{2}^{1/2k}}\dfrac{1}{Z(d^{2k})}=\sum\limits_{\substack{n=1\\n \ square-free}}^{\infty}\dfrac{1}{|Z(n^{2k})|}.
	\end{aligned} $
\end{center}	
Therefore, the left side of the inequality has been established. On the other hand, \begin{center}
	$ \dfrac{|\mathbb{Z}_{W_{2}}^{(k)}|}{|Z(W_{2})|}=\dfrac{|\mathbb{Z}_{W_{1}}^{(k)}|}{|Z(W_{1})|}\cdot\dfrac{|\mathbb{Z}_{W_{2}}^{(k)}|}{|\mathbb{Z}_{W_{1}}^{(k)}|}\cdot\dfrac{|Z(W_{1})|}{|Z(W_{2})|} $.
\end{center}	

Next, we handle $ \dfrac{|\mathbb{Z}_{W_{2}}^{(k)}|}{|\mathbb{Z}_{W_{1}}^{(k)}|} \cdot \dfrac{|Z(W_{1})|}{|Z(W_{2})|} $ .	
\begin{center}
	$ \begin{aligned}
	\dfrac{|\mathbb{Z}_{W_{2}}^{(k)}|}{|\mathbb{Z}_{W_{1}}^{(k)}|} \cdot \dfrac{|Z(W_{1})|}{|Z(W_{2})|}&=\prod\limits_{p|W_{2}}\dfrac{|\mathbb{Z}_{p^{2k}}^{(k)}|}{|\mathbb{Z}_{p^{k}}^{(k)}|}\cdot  \dfrac{|Z(p^{k})|}{|Z(p^{2k})|}\\
	&=\dfrac{1+\frac{2^{2k-1}}{(k,2)}\big(\frac{1}{(k,2^{2k-2})}+\frac{1}{2^{k}(k,2^{k-2})}\big)}{1+\frac{2^{k-1}}{(k,2)(k,2^{k-2})}}\prod\limits_{2<p\leq \omega}\dfrac{1+\frac{p^{2k-1}(p-1)}{(k,p-1)}\big(\frac{1}{(k,p^{2k-1})}+\frac{1}{p^{k}(k,p^{k-1})}\big)}{1+\frac{p^{k-1}(p-1)}{(k,p-1)(k,p^{k-1})}}\\
	&\times \dfrac{2^{k-1}/(k,2)(k,2^{k-2})}{2^{2k-1}/(k,2)(k,2^{2k-2})}\prod\limits_{2<p\leq \omega}\dfrac{p^{k-1}(p-1)/(k,p-1)(k,p^{k-1})}{p^{2k-1}(p-1)/(k,p-1)(k,p^{2k-1})}\\
	&=\bigg(\dfrac{1}{2^{k}}+\dfrac{2^{k-1}}{(k,2)(k,2^{k-2})+2^{k-1}}\bigg)\prod\limits_{2<p\leq \omega}\bigg(\dfrac{1}{p^{k}}+\dfrac{p^{k-1}(p-1)}{(k,p-1)(k,p^{k-1})+p^{k-1}(p-1)}\bigg).
	\end{aligned} $
\end{center}	
Note that \begin{center}
	$ \dfrac{1}{2^{k}}+\dfrac{2^{k-1}}{(k,2)(k,2^{k-2})+2^{k-1}}<1 $
\end{center}	
and
\begin{center}
	$ \dfrac{1}{p^{k}}+\dfrac{p^{k-1}(p-1)}{(k,p-1)(k,p^{k-1})+p^{k-1}(p-1)}<1 $
\end{center}	
for any $ 2<p\leq \omega $.	Therefore, \begin{center}
	$ \dfrac{|\mathbb{Z}_{W_{2}}^{(k)}|}{|Z(W_{2})|}<\dfrac{|\mathbb{Z}_{W_{1}}^{(k)}|}{|Z(W_{1})|} $.
\end{center}
The right side of the inequality follows readily.
\end{proof}

\begin{lemma}\label{lem 4.3}
	Let $ \epsilon \in (0,1) $. Then \begin{center}
		$ \mathbb{E}_{b\in \mathbb{Z}_{W_{2}}^{(k)}}g_{2}(b,N)\geq (1-\epsilon)\delta_{B}^{k} $
	\end{center}
provided that $ N $ is large enough.
\end{lemma}
\begin{proof}
See \cite[proof of Lemma 6.1]{Sal}.
\end{proof}

Using the previous lemmas, we have the following result.
\begin{lemma}\label{lem 4.4}
	Let $ \epsilon\in (0,1) $. Then \begin{center}
		$ \mathbb{E}_{b\in Z(W_{2})}g_{2}(b,N)\geq(1-\epsilon)(\mathcal{T}_{k}\delta_{B}^{k}-\mathcal{Z}_{k}+1) $
	\end{center}
provided that $ N $ is large enough.
\end{lemma}
\begin{proof}
Note that $ g_{2}(b,N)\leq 1+o(1) $. By Lemma \ref{lem 4.2} and Lemma \ref{lem 4.3}, we have \begin{center}
	$ \begin{aligned}
	 \mathbb{E}_{b\in Z(W_{2})}g_{2}(b,N)&=\dfrac{|\mathbb{Z}_{W_{2}}^{(k)}|}{|Z(W_{2})|}\mathbb{E}_{b\in Z_{W_{2}}^{(k)}}g_{2}(b,N)-\dfrac{1}{|Z(W_{2})|}\sum\limits_{\substack{b\in Z_{W_{2}}^{(k)}\\(b,W_{2})>1}}g_{2}(b,N)\\
	 &\geq (1-o(1))\mathcal{T}_{k}\mathbb{E}_{b\in Z_{W_{2}}^{(k)}}g_{2}(b,N)-(1+o(1))\mathcal{Z}_{k}+1\\
	 &\geq (1-o(1))(\mathcal{T}_{k}\delta_{B}^{k}-\mathcal{Z}_{k}+1).
	\end{aligned} $
\end{center}
\end{proof}

Regarding the lower bounds of $ \mathbb{E}_{b \in Z(W_{2})}g_{1}(b,N) $ and $ \mathbb{E}_{b \in Z(W_{2})}\mathbf{g}(b,N) $, we have the following two results.
\begin{lemma}\label{lem 4.5}(\cite[Lemma 4.7]{Gao})
	Let $ \epsilon \in (0,1) $. Then
	\begin{center}
		$ \mathbb{E}_{b \in Z(W_{2})}g_{1}(b,N) \geq k\delta_{A}-(k-1)-\epsilon $
	\end{center}
	provided that $ N $ is large enough depending on $ \epsilon $.	
\end{lemma}
\begin{lemma}\label{lem 4.6}(\cite[Lemma 4.8]{Gao})
	Let $ \epsilon \in (0,1) $. Then
	\begin{center}
		$ \mathbb{E}_{b \in Z(W_{2})}\mathbf{g}(b,N) \geq (1-\epsilon)\delta $
	\end{center}
	provided that $ N $ is large enough depending on $ \epsilon $.	
\end{lemma}

Similar to the proofs of Proposition \ref{pro 3.5}-\ref{pro 3.8}, we can establish the following four results by utilizing Lemma \ref{lem 4.1} and Lemma \ref{lem 4.4}-\ref{lem 4.6}. Note that the conclusions of Proposition \ref{pro 4.8} and Proposition \ref{pro 4.10} do not hold for $ k=4 $.
\begin{proposition}\label{pro 4.7}
	Let $ \epsilon \in (0,1/4) $ and let $ N $ be sufficiently large depending on $ \epsilon $. Let $ \delta_{A}>1-1/2k+2\epsilon/k,\  k\delta_{A}+\mathcal{T}_{k}\delta_{B}^{k}>\mathcal{Z}_{k}+k-1+4\epsilon $ and $ s_{2}\geq 1 ,\ s_{1}\geq 16k\omega(k)+4k+4+s_{2} $. Then, for all $ n\in \mathbb{Z}_{W_{2}} $ with $  n \equiv s_{1}+s_{2} \ (\bmod \ R_{k}) $, there exist $ b_{1},\ldots,b_{s_{1}},b_{s_{1}+1},\ldots,b_{s_{1}+s_{2}} \in Z(W_{2}) $ such that: \\
	(i) $ n\equiv b_{1}+\cdots+b_{s_{1}}+b_{s_{1}+1}+\cdots+b_{s_{1}+s_{2}} (\bmod\ W_{2})$ ;\\
	(ii) $\ g_{1}(b_{i},N)>\epsilon /2 $ for all $ i \in \{1,\ldots,s_{1}\} $ and $\ g_{2}(b_{j},N)>\epsilon/2 $ for all $ j \in \{s_{1}+1,\ldots,s_{1}+s_{2}\} $;\\
	(iii)
	$ g_{1}(b_{1},N)+\cdots+g_{1}(b_{s_{1}},N)+g_{2}(b_{s_{1}+1},N)+\cdots+g_{2}(b_{s_{1}+s_{2}},N)>(s_{1}+s_{2})(1+\epsilon)/2. $
\end{proposition}
\begin{proposition}\label{pro 4.8}
	Let $ \epsilon \in (0,1/6) $ and let $ N $ be sufficiently large depending on $ \epsilon $. Let $ \delta_{B}>((\mathcal{Z}_{k}-1/2+3\epsilon)\mathcal{T}_{k}^{-1})^{1/k} ,\  k\delta_{A}+\mathcal{T}_{k}\delta_{B}^{k}>\mathcal{Z}_{k}+k-1+4\epsilon $ and $ s_{2}\geq 1 ,\ s_{1}\geq 16k\omega(k)+4k+4+s_{2} $. Then, for all $ n\in \mathbb{Z}_{W_{2}} $ with $  n \equiv s_{1}+s_{2} \ (\bmod \ R_{k}) $, there exist $ b_{1},\ldots,b_{s_{1}},b_{s_{1}+1},\ldots,b_{s_{1}+s_{2}} \in Z(W_{2}) $ such that: \\
	(i) $ n\equiv b_{1}+\cdots+b_{s_{1}}+b_{s_{1}+1}+\cdots+b_{s_{1}+s_{2}} (\bmod\ W_{2})$ ;\\
	(ii) $\ g_{2}(b_{i},N)>\epsilon /2 $ for all $ i \in \{1,\ldots,s_{1}\} $ and $\ g_{1}(b_{j},N)>\epsilon/2 $ for all $ j \in \{s_{1}+1,\ldots,s_{1}+s_{2}\} $;\\
	(iii)
	$ g_{2}(b_{1},N)+\cdots+g_{2}(b_{s_{1}},N)+g_{1}(b_{s_{1}+1},N)+\cdots+g_{1}(b_{s_{1}+s_{2}},N)>(s_{1}+s_{2})(1+\epsilon)/2. $
\end{proposition}	
\begin{proposition}\label{pro 4.9}
	Let $ \epsilon \in (0,1/6) $ and let $ N $ be sufficiently large depending on $ \epsilon $. Let $ \delta>1/2+3\epsilon,\  \delta+\mathcal{T}_{k}\delta_{B}^{k}>\mathcal{Z}_{k}+4\epsilon $ and $ s_{2}\geq 1 ,\ s_{1}\geq 16k\omega(k)+4k+4+s_{2} $. Then, for all $ n\in \mathbb{Z}_{W_{2}} $ with $  n \equiv s_{1}+s_{2} \ (\bmod \ R_{k}) $, there exist $ b_{1},\ldots,b_{s_{1}},b_{s_{1}+1},\ldots,b_{s_{1}+s_{2}} \in Z(W_{2}) $ such that: \\
	(i) $ n\equiv b_{1}+\cdots+b_{s_{1}}+b_{s_{1}+1}+\cdots+b_{s_{1}+s_{2}} (\bmod\ W_{2})$ ;\\
	(ii) $\ \mathbf{g}(b_{i},N)>\epsilon /2 $ for all $ i \in \{1,\ldots,s_{1}\} $ and $\ g_{2}(b_{j},N)>\epsilon/2 $ for all $ j \in \{s_{1}+1,\ldots,s_{1}+s_{2}\} $;\\
	(iii)
	$ \mathbf{g}(b_{1},N)+\cdots+\mathbf{g}(b_{s_{1}},N)+g_{2}(b_{s_{1}+1},N)+\cdots+g_{2}(b_{s_{1}+s_{2}},N)>(s_{1}+s_{2})(1+\epsilon)/2. $
\end{proposition}
\begin{proposition}\label{pro 4.10}
	Let $ \epsilon \in (0,1/6) $ and let $ N $ be sufficiently large depending on $ \epsilon $. Let $ \delta_{B}>((\mathcal{Z}_{k}-1/2+3\epsilon)\mathcal{T}_{k}^{-1})^{1/k},\  \delta+\mathcal{T}_{k}\delta_{B}^{k}>\mathcal{Z}_{k}+4\epsilon  $ and $ s_{2}\geq 1 ,\ s_{1}\geq 16k\omega(k)+4k+4+s_{2} $. Then, for all $ n\in \mathbb{Z}_{W_{2}} $ with $  n \equiv s_{1}+s_{2} \ (\bmod \ R_{k}) $, there exist $ b_{1},\ldots,b_{s_{1}},b_{s_{1}+1},\ldots,b_{s_{1}+s_{2}} \in Z(W_{2}) $ such that: \\
	(i) $ n\equiv b_{1}+\cdots+b_{s_{1}}+b_{s_{1}+1}+\cdots+b_{s_{1}+s_{2}} (\bmod\ W_{2})$ ;\\
	(ii) $\ g_{2}(b_{i},N)>\epsilon /2 $ for all $ i \in \{1,\ldots,s_{1}\} $ and $\ \mathbf{g}(b_{j},N)>\epsilon/2 $ for all $ j \in \{s_{1}+1,\ldots,s_{1}+s_{2}\} $;\\
	(iii)
	$ g_{2}(b_{1},N)+\cdots+g_{2}(b_{s_{1}},N)+\mathbf{g}(b_{s_{1}+1},N)+\cdots+\mathbf{g}(b_{s_{1}+s_{2}},N)>(s_{1}+s_{2})(1+\epsilon)/2. $
\end{proposition}

\subsection{Conclusions}
\begin{proposition}\label{pro 4.11}(\cite[Proposition 5.1]{Gao})
	Let $ \alpha \in \mathbb{T} $ . For $ b\in [W_{2}] $ with $ b \in Z(W_{2}) $, we have
	\begin{center}
		$ \mid \widehat{\nu_{b}^{\prime}}(\alpha) - \widehat{1_{[N]}}(\alpha) \mid =o(N). $
	\end{center}
\end{proposition}

\begin{proposition}\label{pro 4.12}
	Let $ \alpha \in \mathbb{T} $ . For $ b\in [W_{2}] $ with $ b \in Z(W_{2}) $, we have
	\begin{center}
		$ \mid \widehat{\nu_{b}^{\prime\prime}}(\alpha) - \widehat{1_{[N]}}(\alpha) \mid =o(N). $
	\end{center}
\end{proposition}
\begin{proof}
	See \cite[Section 7]{Sal}.
\end{proof}

\begin{proposition}\label{pro 4.13}(\cite[Proposition 6.3]{Gao})
	Let $ s\in \mathbb{N} $ and $ s>k(k+1) $. For $ b\in [W_{2}] $ with $ b\in Z(W_{2}) $ and $ q\in (s-1,s) $ , we have \begin{center}
		$ \| \widehat{f_{b}^{\prime}} \|_{q} \ll N^{1-1/q} $.
	\end{center}	
\end{proposition}
\begin{proposition}\label{pro 4.14}
	Let $ s\in \mathbb{N} $ and $ s>k(k+1) $. For $ b\in [W_{2}] $ with $ b\in Z(W_{2}) $ , there exists $ q\in (s-1,s) $ such that \begin{center}
		$ \| \widehat{f_{b}^{\prime\prime}} \|_{q} \ll N^{1-1/q} $.
	\end{center}	
\end{proposition}
\begin{proof}
	See \cite[Section 8]{Sal}.
\end{proof}
\begin{proposition}\label{pro 4.15}(\cite[Proposition 6.4]{Gao})
	Let $ s>k(k+1) $. For $ b\in [W_{2}] $ with $ b\in Z(W_{2}) $ and  $ q\in (s-1,s) $ , we have \begin{center}
		$ \| \widehat{\mathbf{f}}_{b} \|_{q} \ll N^{1-1/q} $.
	\end{center}	
\end{proposition}

$ \mathit{Proof \ of \ Theorem \ \ref{theorem 1.5}-\ref{theorem 1.8}} $. The proofs of Theorem \ \ref{theorem 1.5}-\ref{theorem 1.8} are omitted, since they follow directly by repeating the proofs of Theorem \ \ref{theorem 1.1}-\ref{theorem 1.4} with Propositions \ref{pro 4.7}-\ref{pro 4.15} in place of Propositions \ref{pro 3.5}-\ref{pro 3.13} . \   \quad   $ \qedsymbol $

\section*{Acknowledgements}
The author would like to thank Professor Yonghui Wang for constant encouragement and for many valuable guidance, and thank Wenying Chen for helpful discussions. 

\bibliographystyle{amsplain}

\end{document}